\documentclass[submission]{FPSAC2021}

\usepackage{amsthm}

\newtheorem{teo}{Theorem}[section]
\newtheorem{cor}[teo]{Corollary}

\newtheorem{lema}[teo]{Lemma}

\theoremstyle{definition}
\newtheorem{alg}[teo]{Algorithm}
\newtheorem{defin}[teo]{Definition}
\newtheorem{obs}[teo]{Remark}
\theoremstyle{remark}
\newtheorem{ex}[teo]{Example}

\usepackage{tikz}
\usepackage{lipsum}

\newcommand{\YT}[3]{
	\vcenter{\hbox{
			\begin{tikzpicture}[x={(0in,-#1)},y={(#1,0in)}] 
			\foreach \rowi [count=\i] in {#3} {
				\foreach \e [count=\j] in \rowi {
					\draw (\i,\j) rectangle +(-1,-1);
					\draw (\i-0.5,\j-0.5) node {$#2\e$};
				}
			}
			\end{tikzpicture}
		}}
	}
	
	\newcommand{\SYT}[3]{
		\vcenter{\hbox{
				\begin{tikzpicture}[x={(0in,-#1)},y={(#1,0in)}] 
				\foreach \rowi [count=\i] in {#3} {
					\foreach \e [count=\j] in \rowi {
						\draw (\i,-\j) rectangle +(-1,-1);
						\draw (\i-0.5,-\j-0.5) node {$#2\e$};
					}
				}
				\end{tikzpicture}
			}}
		}

\title[Type $C$ Willis' direct way]{Symplectic right keys - Type $C$ Willis' direct way}

\author[J. M. Santos]{João Miguel Santos\thanks{\href{jmsantos@mat.uc.pt}{jmsantos@mat.uc.pt}. This work was partially supported by the Center for Mathematics of the  University of Coimbra - UID/MAT/00324/2019, funded by the Portuguese Government through FCT/MEC and co-funded by the European Regional Development Fund through the Partnership Agreement PT2020. It was also supported by FCT, through the grant PD/BD/142954/2018, under POCH funds, co-financed by the European Social Fund and Portuguese National Funds from MEC.}}

\address{CMUC, Department of Mathematics, University of Coimbra, Apartado 3008,
	3001--454 Coimbra, Portugal}

\received{\today}

\abstract{It is known that the right key of a Kashiwara-Nakashima tableau can be computed using the Lecouvey-Sheats symplectic \emph{jeu de taquin}. Motivated by Willis' direct way of computing type $A$ right keys, we also give a way of computing symplectic right keys without the use of \emph{jeu de taquin}.	
}

\resume{En type C, la clé droite d'un tableau de Kashiwara-Nakashima peut se calculer
	  en utilisant le \emph{jeu de taquin} symplectique de Lecouvey-Sheats. Motivés par la manière directe de Willis afin de calculer les clés droites du type $A$, nous proposons également une méthode de calculer les clés droites symplectiques sans utiliser le \emph{jeu de taquin}.}

\keywords{Direct way, Symplectic keys, Right key map, Symplectic \emph{jeu de taquin}}

\usepackage[backend=bibtex]{biblatex}
\begin{document}

\maketitle

\section{Introduction}
Symplectic tableaux provide the monomial weight generators for the characters of the symplectic Lie algebra $sp(2n, \mathbb{C})$. Given a partition $\lambda$, symplectic Kashiwara-Nakashima tableaux \cite{Kashiwara1994CrystalGF}, a variation of De Concini tableaux \cite{de1979symplectic}, of shape $\lambda$ are endowed with a crystal structure $\mathfrak{B}_\lambda$ compatible with a plactic monoid  and sliding algorithms, studied by Lecouvey in terms of crystal isomorphisms \cite{lecouvey2002schensted}.
Let $W\lambda$ be the orbit of $\lambda$,  where $W$ is the type $C_n$ Weyl group. Type $C_n$ Demazure characters are indexed by vectors $v$ in $W\lambda$ and can be seen as "partial" characters. Kashiwara \cite{kashiwara1992crystal} and Littelmann \cite{littelmann1995crystal} have shown that they can be obtained by summing the monomial weights over certain subsets $\mathfrak{B}_v$ in the crystal $\mathfrak{B}_\lambda$, called Demazure crystals. Demazure crystals $\mathfrak{B}_v$ can be partitioned into Demazure atom crystals, $\widehat{\mathfrak{B}}_u$, where $u\in W\lambda$ runs in the Bruhat interval $\lambda\le u\le v$. 


In type $A_{n-1}$, Lascoux and Sch\"{u}tzenberger characterized key tableaux tableaux as semistandard Young tableaux (SSYT) with nested columns \cite{lascoux1990keys}, and have used the \emph{jeu de taquin} to define the right key map which sends a SSYT to a key tableau, called the right key of that SSYT. In each Demazure atom crystal there exists exactly one key tableau and the right key detects the Demazure atom crystal that contains a given SSYT \cite[Theorem $3.8$]{lascoux1990keys}. 
By direct inspection of a Young tableau, Willis \cite{Willis2011ADW} has  given an alternative algorithm to compute the right key tableau that does not require the use of \emph{jeu de taquin}. 
Other methods to compute the type $A$ right key map includes the alcove path model \cite{lenart2015generalization}, semi skyline augmented fillings \cite{Mason2007AnEC}, and coloured vertex models. For a complete overview in type $A$, see \cite{BRUBAKER2021105354} and the references therein. In type $C_n$, the symplectic key tableaux are characterized in \cite{azenhas2012key, santos2019symplectic, santos2020SLCsymplectic,jacon2019keys}. Using the Sheats symplectic \emph{jeu de taquin}, a right key map is given, in \cite{santos2019symplectic, santos2020SLCsymplectic,jacon2019keys}, to send a Kashiwara-Nakashima (KN) tableau $T$ to its right key tableau $K_+(T)$, that detects the Demazure atom crystal which contains $T$.
They are also computed in the type $C_n$ alcove path model \cite{lenart2015generalization}, as well as for reverse King tableau using the colored five vertex model \cite{buciumas2021quasi}.

Jacon and Lecouvey \cite{jacon2019keys}  have suggested that Willis' method should be adaptable to type $C_n$.   Motivated by Willis' direct inspection \cite{Willis2011ADW}, we create an alternative algorithm, based on the direct inspection of a KN tableau, for the symplectic right key map, that does not use the symplectic \emph{jeu de taquin}.
Due to the added technicality of the symplectic \emph{jeu de taquin} compared to the one for SSYT, Willis' \emph{earliest weakly increasing subsequence} will fail to predict what gets slid during the Sheats symplectic \emph{jeu de taquin}. Instead we need a way to calculate, without the use of \emph{jeu de taquin}, what would appear in each column if we were to swap its length with the previous column length via \emph{jeu de taquin}. The role of Willis' sequences will be replaced by our matchings (see Section \ref{rightkeydw}). In type $A$, these kind of matches were used earlier \cite{azenhas2006schur}.

The paper is organized as follows. In Section \ref{tabjdt}, we discuss the  type $C$ Kashiwara-Nakashima tableaux and the symplectic \emph{jeu de taquin}. 
Section \ref{rightkeyjdt} briefly recalls the symplectic key tableaux and right key map via symplectic \emph{jeu de taquin} \cite{santos2019symplectic}.
In Section \ref{rightkeydw}, we give an algorithm for computing  the symplectic right key map that does not require the \emph{jeu de taquin}, and prove that it returns the same object as the previous method.

\textit{This is an extended abstract of a full paper to appear.}
\section{Type $C$ Kashiwara-Nakashima tableaux and \emph{jeu de taquin}}\label{tabjdt}

We recall the symplectic tableaux introduced by Kashiwara and Nakashima to label the vertices of the type $C_n$ crystal graphs \cite{Kashiwara1994CrystalGF}.
Fix $n\in \mathbb{N}_{>0}$. Define the sets $[n]=\{1, \dots, n\}$ and $[\pm n]=\{1,\dots, n, \overline{n}, \dots, \overline{1}\}$ where $\overline{i}$ is just another way of writing $-i$, hence $\overline{\overline{i}}=i$. In the second set we will consider the following order of its elements: $1<\dots< n< \overline{n}< \dots< \overline{1}$ instead of the usual order.
A vector $\lambda=(\lambda_1,\dots, \lambda_n)\in\mathbb{Z}^n$ is a partition of $|\lambda|=\sum\limits_{i=1}^n \lambda_i$ if $\lambda_1\geq \lambda_2\geq\dots\geq\lambda_n\geq 0$.
The \emph{Young diagram} of shape $\lambda$, in English notation, is an array of boxes (or cells), left justified, in which the $i$-th row, from top to bottom, has $\lambda_i$ boxes. We identify a partition with its Young diagram.
For example, the Young diagram of shape $\lambda=(2,2,1)$ is $\YT{0.17in}{}{
	{{},{}},
	{{},{}},
	{{}}}$.
Given $\mu$ and $\nu$ two partitions with $\nu\leq \mu$ entrywise, we write $\nu\subseteq \mu$. The Young diagram of shape $\mu/\nu$ is obtained after removing the boxes of the Young diagram of $\nu$ from the Young diagram of $\mu$.
For example, the Young diagram of shape $\mu/\nu=(2,2,1)/(1,0,0)$ is $\begin{tikzpicture}[scale=.4, baseline={([yshift=-.8ex]current bounding box.center)}]
\draw (1,0) rectangle +(-1,-1);
\draw (1,-1) rectangle +(-1,-1);
\draw (0,-2) rectangle +(-1,-1);
\draw (0,-1) rectangle +(-1,-1);
\end{tikzpicture}$\,.
	Let $\nu\subseteq \mu$ be two partitions and $A$ a completely ordered alphabet. A \emph{semistandard Young tableau}, for short SSYT, of skew shape $\mu/\nu$ on the alphabet $A$ is a filling of the diagram $\mu/\nu$ with letters from $A$, such that the entries are strictly increasing, from top to bottom, in each column and weakly increasing, from left to right, in each row. When $|\nu|=0$ then we obtain a SSYT of shape $\mu$.
Denote by $\text{SSYT}(\mu/\nu, A)$
the set of all skew SSYT $T$ of shape $\mu/\nu$, with entries in $A$. When $|v|=0$ we write $\text{SSYT}(\mu, A)$ and when $A=[n]$ we write $\text{SSYT}(\mu/\nu, n)$.
When considering tableaux with entries in $[\pm n]$, it is usual to have some extra conditions besides being semistandard. We will use a family of tableaux known as \emph{Kashiwara-Nakashima} tableaux.
From now on we consider tableaux on the alphabet $[\pm n]$.

A \emph{column} is a strictly increasing sequence of numbers (or letters) in $[\pm n]$ and it is usually displayed vertically. The heigth of a column ifs number of letters in it.
A column is said to be \emph{admissible} if the following \emph{one column condition} (1CC) holds for that column:

\begin{defin}[1CC]\label{1CC}
	Let $C$ be a column. The $1CC$ holds for $C$ if for all pairs $i$ and $\overline{i}$ in $C$, where $i$ is in the $a$-th row counting from the top of the column, and $\overline{i}$ in the $b$-th row counting from the bottom, we have $a+b\leq i$.
\end{defin}
If a column $C$ satisfies the $1CC$ then $C$ has at most $n$ letters.
If $1CC$ doesn't hold for $C$ we say that $C$ \emph{breaks the $1CC$ at $z$}, where $z$ is the minimal such that $z$ and $\overline{z}$ exist in $C$ and there are more than $z$ numbers in $C$ with absolute value less or equal than $z$. 
For instance, the column $\YT{0.17in}{}{
		{{1}},
		{{2}},
		{{\overline{1}}}}$ breaks the $1CC$ at $1$, and $\YT{0.17in}{}{
		{{2}},
		{{3}},
		{{\overline{3}}}}$ is an admissible column.
The following definition states conditions to when $C$ can be \emph{split}:
\begin{defin}
	Let $C$ be a column and let $I = \{z_1 > \dots > z_r\}$ be the
	set of unbarred letters $z$ such that the pair $(z, \overline{z})$ occurs in $C$. The column
	$C$ can be split when there exists a set of $r$ unbarred
	letters $J = \{t_1 > \dots > t_r\} \subseteq [n]$ such that:

\boldmath{1.} $t_1$ is the greatest letter of $[n]$ satisfying $t_1 < z_1$,  $t_1 \not\in C$, and $\overline{t_1}\not\in C$,

\boldmath{2.} for $i=2, \dots, r$, we have that  $t_i$ is the greatest letter of $[n]$ satisfying $t_i < \min(t_{i-1},   z_i)$,  $t_i \not\in C$, and $\overline{t_i} \not\in C$.
\end{defin}
	The column $C$ is admissible if and only if $C$ can be split  \cite[Lemma 3.1]{sheats1999symplectic}.
	If $C$ can be split then we define \emph{right column} of $C$, $\text{r}C$, and the \emph{left column} of $C$, $\ell C$, as follows:
	
	\boldmath{1.} $rC$ is obtained by changing in $C$,  $\overline{z_i}$ into $\overline{t_i}$ for each $z_i \in I $ and reordering,
		
	\boldmath{2.} $\ell C$ is obtained by changing in $C$, $z_i$ into $t_i$ for each  $z_i \in I $ and reordering.

	If $C$ is admissible then $\ell C\leq C \leq rC$ by entrywise comparison, where $\ell C$ has the same barred part as $C$ and $rC$ the same unbarred part. If $C$ doesn't have symmetric entries, then $C$ is admissible and  $\ell C= C =rC$.
	In the next definition we give conditions for a column $C$ to be \emph{coadmissible}.
	\begin{defin}
		We say that a column $C$ is coadmissible if for every pair $i$ and $\overline{i}$ on $C$, where $i$ is on the $a$-th row counting from the top of the column, and $\overline{i}$ on the $b$-th row counting from the top, then $b-a\leq n-i$.
	\end{defin}
\noindent
	Given an admissible column $C$, consider the map $\Phi$ that sends $C$ to the column $\Phi(C)$ of the same size in which the unbarred entries are taken from $\ell C$ and the barred entries are taken from $rC$. The \emph{column $\Phi(C)$} is coadmissible  and the algorithm to form $\Phi(C)$ from $C$ is reversible \cite[Section 2.2]{lecouvey2002schensted}. In particular, every column without symmetric entries is simultaneously admissible and coadmissible. 
	
	\begin{ex}Let $n=4$ and
		$C=\!\YT{0.17in}{}{
			{{2}},
			{{4}},
			{{\overline{2}}}}$ be an admissible column. Then $\ell C=\YT{0.17in}{}{
			{{1}},
			{{4}},
			{{\overline{2}}}}$ and $rC=\YT{0.17in}{}{
			{{2}},
			{{4}},
			{{\overline{1}}}}$. So $\Phi(C)=\YT{0.17in}{}{
			{{1}},
			{{4}},
			{{\overline{1}}}}$ is coadmissible. $C$ is also coadmissible and $\Phi^{-1}(C)=\YT{0.17in}{}{
			{{3}},
			{{4}},
			{{\overline{3}}}}$.
	\end{ex}
	
	Let $T$ be a skew tableau with all of its columns admissible. The \emph{split form} of a skew tableau $T$, $spl(T)$, is the skew tableau obtained after replacing each column $C$ of $T$ by the two columns $\ell C\,rC$. The tableau $spl(T)$ has double the amount of columns of $T$.

		A skew SSYT $T$ is a \emph{Kashiwara-Nakashima (KN) skew tableau} if its split form is a skew SSYT. We define $\mathcal{KN}(\mu/\nu, n)$ to be the set of all KN tableaux of shape $\mu/\nu$ in the alphabet $[\pm n]$. When $\nu=0$, we obtain $\mathcal{KN}(\mu, n)$. The weight of $T$ is a vector whose $i$-th entry is the number of $i$'s minus the number of $\overline{i}$.

	\begin{ex} The split of the tableau 
		$T=\YT{0.17in}{}{
			{{2},{2}},
			{{3},{3}},
			{{\overline{3}}}}$ is the tableau $spl(T)=\YT{0.17in}{}{
			{{1}, {2},{2},{2}},
			{{2},{3},{3},{3}},
			{{\overline{3}},{{\overline{1}}}}}$. Hence $T\in \mathcal{KN}((2,2,1),3)$ and weight $\text{wt} T=(0,2,1)$.
	\end{ex}
	
	If $T$ is a tableau without symmetric entries in any of its columns, i.e., for all $i\in[n]$ and for all columns $C$ in $T$, $i$ and $\overline{i}$ do not appear simultaneously in the entries of $C$, then in order to check whether $T$ is a KN tableau it is enough to check whether $T$ is semistandard in the alphabet $\left[\pm n\right]$. In particular $SSYT(\mu/\nu, n)\subseteq \mathcal{KN}(\mu/\nu, n)$.
	
	\subsection{Sheats symplectic \textit{jeu de taquin}}
Sheats symplectic \textit{jeu de taquin} (SJDT) \cite{lecouvey2002schensted, sheats1999symplectic} is a procedure on KN skew tableaux, compatible with \emph{Knuth equivalence} (or plactic equivalence on words over the alphabet $[\pm n]$) \cite{lecouvey2002schensted}, that allows us to change the shape of a tableau and to rectify it.   
	To explain how the SJDT behaves, we need to look how it works on $2$-column $C_1C_2$ KN skew tableaux. A skew tableau is \emph{punctured} if one of its box contains
	the symbol $\ast$ called the \emph{puncture}.
	A punctured column
	is admissible if the column is admissible when ignoring the puncture. A punctured skew tableau is
	admissible if its columns are admissible and the rows of its split form are weakly increasing ignoring
	the puncture. Let $T$ be a punctured skew tableau with two columns $C_1$ and $C_2$ with the puncture
	in $C_1$. In that
	case, the puncture splits into two punctures in $spl(T)$, and ignoring the punctures, $spl(T)$ must be semistandard. 
	Let $\alpha$ be the entry under the puncture of $rC_1$, and $\beta$ the entry to the right of the puncture of $rC_1$ ($\alpha$ or $\beta$ may not exist):
	$spl(T)=\ell C_1 rC_1\ell C_2rC_2=\YT{0.2in}{}{
		{{\dots},{\dots}, {\dots}, {\dots}},
		{{\ast},{\ast},{\beta},{\dots}},
		{{\dots},{\alpha},{\dots},{\dots}},
		{{\dots},{\dots}}}.$
	
\noindent The elementary steps of SJDT are the following:

 \textbf{A.} If $\alpha\leq \beta$ or $\beta$ does not exist,  then the puncture of $T$ will change its position with the cell beneath it. This is a vertical slide.

 \textbf{B.}	If the slide is not vertical, then it is horizontal. So we have $\alpha> \beta$ or $\alpha$ does not exist. Let $C_1'$ and $C_2'$ be the columns obtained after the slide. We have two subcases, depending on the sign of $\beta$:

	\boldmath{1.} If $\beta$ is barred, we are moving a barred letter, $\beta$, from $\ell C_2$ to the punctured box of $rC_1$, and the puncture will occupy $\beta$'s place in $\ell C_2$. Note that $\ell C_2$ has the same barred part as $C_2$ and that $rC_1$ has the same barred part as $\Phi(C_1)$. Looking at $T$, we will have an horizontal slide of the puncture, getting $C_2'=C_2\setminus \{\beta\}\sqcup\{\ast\}$ and $C_1'=\Phi^{-1}(\Phi(C_1)\setminus{\ast}\sqcup \{\beta\})$. In a sense, $\beta$ went from $C_2$ to $\Phi(C_1)$.
		
	\boldmath{2.} If $\beta$ is unbarred, we have a similar story, but this time $\beta$ will go from $\Phi(C_2)$ to $C_1$, hence $C_1'=C_1\setminus{\ast}\cup \{\beta\}$ and $C_2'=\Phi^{-1}(\Phi(C_2)\setminus \{\beta\}\sqcup{\ast})$. Although in this case it may happen that $C_1'$ is no longer admissible. In this situation, if the 1CC breaks at $i$, we erase both $i$ and $\overline{i}$ from the column and remove a cell from the bottom and from the top column, and place all the remaining cells orderly.
	
	Applying successively elementary SJDT steps, eventually, the puncture will be a cell such that $\alpha$ and $\beta$ do not exist. Then, we redefine the shape to not include this cell and SJDT ends. 
	Given an admissible tableau $T$ of shape $\mu/\nu$, a box of the diagram of shape $\nu$ such that boxes under it and to the right are not in that shape is called an \emph{inner corner} of $\mu/\nu$. An outside corner is a box of $\mu$ such that boxes under it and to the right are not in the shape $\mu$. The rectification of $T$ consists in playing the SJDT until we get a tableau of shape $\lambda$, for some partition $\lambda$. More precisely, apply successively elementary SJDT steps
	to $T$ until each cell of $\nu$ becomes an \emph{outside corner}. At the end, we obtain a KN tableau for some
	shape $\lambda$. The rectification is independent of the order in which the inner corners of $\nu$ are filled \cite[Corollary 6.3.9]{lecouvey2002schensted}. 
	
 Consider the KN skew tableau
		$T=\SYT{0.145in}{}{
			{{2}},
			{{3},{1}},
			{{\overline{1}},{2}}}$.
		To rectify $T$ via SJDT, one creates a puncture in a inner corner and split, obtaining $\YT{0.17in}{}{
			{{\ast}, {\ast},{2},{2}},
			{{1},{1},{3},{3}},
			{{2},{2},{\overline{1}},{\overline{1}}}}$. So, the first two slides are vertical, obtaining $\YT{0.17in}{}{
			{{1}, {1},{2},{2}},
			{{2},{2},{3},{3}},
			{{\ast},{\ast},{\overline{1}},{\overline{1}}}}$. Finally, we do an horizontal slide, of type $B.1$, in which we take $\overline{1}$ from the second column of $T$ and add it to the coadmissible column of the first column of $T$, obtaining the tableau
		$\YT{0.145in}{}{
			{{2},{2}},
			{{3},{3}},
			{{\overline{3}}}}$.
Let $T$ be a skew tableau of shape $\mu/\nu$. Consider a punctured box that can be added to $\mu$, so that $\mu\cup \{\ast\}$ is a valid shape.
The SJDT is reversible, meaning that we can move $\ast$, the empty cell outside of $\mu$, to the inner shape $\nu$ of the skew tableau $T$, simultaneously increasing both the inner and outer shapes of $T$ by one cell. The slides work similarly to the previous case: the vertical slide means that an empty cell is going up and an horizontal slide means that an entry goes from $\Phi(C_1)$ to $C_2$ or from $C_1$ to $\Phi(C_2)$, depending on whether the slid entry is barred or not, respectively.
We will also call the \emph{reverse jeu de taquin} as SJDT. In the next sections we will be mostly dealing with the \emph{reverse jeu de taquin}.
Consider the following examples, each containing a tableau and a punctured box that will be slid to its inner shape:
$\YT{0.17in}{}{
	{{},{},{\ast}},
	{{1},{\overline{1}}},
	{{2}}} \mapsto \YT{0.17in}{}{
	{{},{},{}},
	{{1},{\overline{1}}},
	{{2}}}$;\quad
$\YT{0.17in}{}{
	{{},{}},
	{{1},{\overline{1}}},
	{{2},{\ast}}}\mapsto \YT{0.17in}{}{
	{{},{}},
	{{},{2}},
	{{2},{\overline{2}}}}$.
	
		If columns $C_1$ and $C_2$ do not have negative entries then the SJDT applied to $C_1C_2$ coincides with the \textit{jeu de taquin} known for SSYT. 
		Next section, we use SJDT to swap lengths of consecutive columns in a skew tableau,  to obtain skew tableaux Knuth related to a straight tableau, which is minimal for the number of cells. Hence, SJDT will not incur in a loss/gain of boxes, that could happen in the elementary step $B.2$.
	

	\section{The right key of a tableau - \emph{Jeu de taquin} approach}\label{rightkeyjdt}
	
	A \emph{key tableau} of shape $\lambda$, in type $C_n$, is a KN tableau in $\mathcal{KN}(\lambda, n)$, in which the set of elements of
	each column, right to left, contains the set of elements of the previous column, if any, and the letters $i$ and $\overline{i}$ do not appear simultaneously as entries, for any $i \in[n]$. The
	split form of a KN key tableau consists of the duplication of each column.
	
	The set $\mathcal{KN}(\lambda, n)$ is endowed with a symplectic crystal structure, denoted $\mathfrak{B}_\lambda$ \cite{Kashiwara1994CrystalGF, lecouvey2002schensted}. The key tableaux in $\mathcal{KN}(\lambda, n)$ are the unique tableaux in $\mathcal{KN}(\lambda, n)$ whose weight is $\alpha \lambda$, for all elements $\alpha$ in the Weyl group $W$, denoted $key(\alpha \lambda)$. The orbit of $key(\lambda)$, the highest weight element of $\mathfrak{B}_\lambda$, under the action of the Weyl group $W$, is defined to be $O(\lambda) = \{\text{key}(\alpha \lambda) : \alpha \in W\}$.
	The \emph{right key map} $K_+$ sends each $T\in \mathcal{KN}(\lambda, n)$ to one element $K_+(T)$ in the
	orbit $O(\lambda)$, called the \emph{right key tableau} of $T$. The right key of a
	KN tableau $T$ is a key tableau of the same shape as $T$, entrywise ”slightly” bigger than $T$. In \cite{santos2019symplectic, santos2020SLCsymplectic}
	such tableau is computed using the aforementioned SJDT.
	\begin{lema}\label{sahpesjdt}\cite{santos2019symplectic, santos2020SLCsymplectic}
		Given $T\in \mathcal{KN}(\lambda, n)$ and a skew shape whose column lengths are a permutation of the column lengths of $T$, then there is exactly one skew tableau with that shape that rectifies to $T$. Futhermore, the last column only depends on its length.
	\end{lema}
	\begin{defin}[Right key map]\cite{santos2019symplectic, santos2020SLCsymplectic}\label{rightkeymap}
		Given $T\in \mathcal{KN}(\lambda, n)$, we 
		consider the KN skew tableaux with the same number of columns of each length as $T$, each one corresponding to a permutation of its column lengths. 
		Then we replace each column of $T$ with a column of the same size taken from the right columns of the last columns of all those skew tableaux associated to $T$. This tableau is the right key tableau of $T$, $K_+(T)$. This map restricted to $SSYT(\lambda,n)$ recovers the Lascoux-Schützenberger right key map \cite{lascoux1990keys}.
	\end{defin}
Given $T\in \mathcal{KN}(\lambda, n)$ we apply the SJDT on consecutive columns to compute all skew tableaux in the conditions the previous definition. For instance, the tableau $T=\YT{0.17in}{}{
			{{1},{3},{\overline{1}}},
			{{3},{\overline{3}}},
			{{\overline{3}}}}$ gives rise to six KN skew tableaux with the same number of columns of each length as $T$, each one corresponding to a permutation of its column lengths.
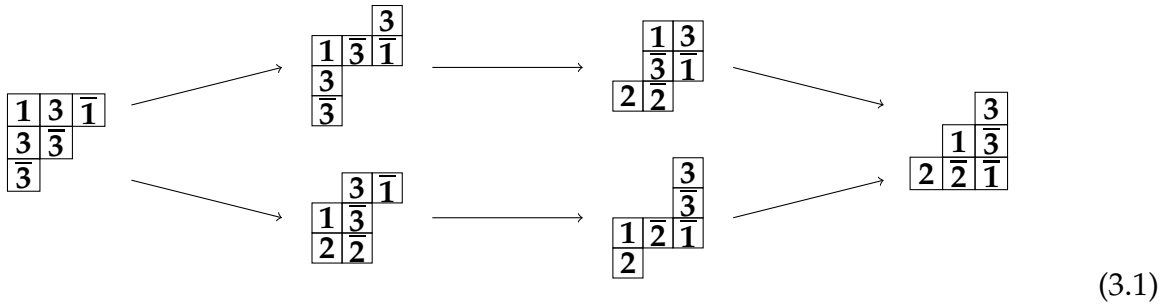
\begin{equation}\label{jeudetaquin3}		\begin{tikzpicture}
		\node at (0,0) {$\YT{0.17in}{}{
				{{1},{3},{\overline{1}}},
				{{3},{\overline{3}}},
				{{\overline{3}}}}$};	
		\node at (4,1) {$\begin{tikzpicture}[scale=.4, baseline={([yshift=-.8ex]current bounding box.center)}]
			\draw (0,0) rectangle +(1,1);
			\draw (0,1) rectangle +(1,1);
			\draw (0,2) rectangle +(1,1);
			\draw (1,2) rectangle +(1,1);
			\draw (2,2) rectangle +(1,1);
			\draw (2,3) rectangle +(1,1);
			\node at (.5,.5) {$\overline{3}$};
			\node at (.5,1.5) {$3$};
			\node at (.5,2.5) {$1$};
			\node at (1.5,2.5) {$\overline{3}$};
			\node at (2.5,2.5) {$\overline{1}$};
			\node at (2.5,3.5) {$3$};
			\end{tikzpicture}$};
		\node at (8,1) {$\begin{tikzpicture}[scale=.4, baseline={([yshift=-.8ex]current bounding box.center)}]
			\draw (0,0) rectangle +(1,1);
			\draw (1,0) rectangle +(1,1);
			\draw (1,1) rectangle +(1,1);
			\draw (1,2) rectangle +(1,1);
			\draw (2,1) rectangle +(1,1);
			\draw (2,2) rectangle +(1,1);
			\node at (.5,.5) {$2$};
			\node at (1.5,.5) {$\overline{2}$};
			\node at (1.5,1.5) {$\overline{3}$};
			\node at (1.5,2.5) {$1$};
			\node at (2.5,1.5) {$\overline{1}$};
			\node at (2.5,2.5) {$3$};
			\end{tikzpicture}$};
		\node at (4,-1) {$\begin{tikzpicture}[scale=.4, baseline={([yshift=-.8ex]current bounding box.center)}]
			\draw (0,0) rectangle +(1,1);
			\draw (0,1) rectangle +(1,1);
			\draw (1,0) rectangle +(1,1);
			\draw (1,1) rectangle +(1,1);
			\draw (1,2) rectangle +(1,1);
			\draw (2,2) rectangle +(1,1);
			\node at (.5,.5) {$2$};
			\node at (.5,1.5) {$1$};
			\node at (1.5,.5) {$\overline{2}$};
			\node at (1.5,1.5) {$\overline{3}$};
			\node at (1.5,2.5) {$3$};
			\node at (2.5,2.5) {$\overline{1}$};
			\end{tikzpicture}$};
		\node at (8,-1) {$$\begin{tikzpicture}[scale=.4, baseline={([yshift=-.8ex]current bounding box.center)}]
			\draw (0,0) rectangle +(1,1);
			\draw (0,1) rectangle +(1,1);
			\draw (1,1) rectangle +(1,1);
			\draw (2,1) rectangle +(1,1);
			\draw (2,2) rectangle +(1,1);
			\draw (2,3) rectangle +(1,1);
			\node at (.5,.5) {$2$};
			\node at (.5,1.5) {$1$};
			\node at (1.5,1.5) {$\overline{2}$};
			\node at (2.5,1.5) {$\overline{1}$};
			\node at (2.5,2.5) {$\overline{3}$};
			\node at (2.5,3.5) {$3$};
			\end{tikzpicture}$$};
		\node at (12,0) {$\SYT{0.17 in}{}{{{3}},{{\overline{3}},{1}},{{\overline{1}},{\overline{2}},{2}}}$};
		\draw [->] (1,.5) -- (3,1);
		\draw [->] (1,-.5) -- (3,-1);
		\draw [->] (5,1) -- (7,1);
		\draw [->] (5,-1) -- (7,-1);
		\draw [->] (9,1) -- (11,.5);
		\draw [->] (9,-1) -- (11,-.5);	
		\end{tikzpicture}\end{equation}
		
		\noindent The right key tableau of $T$ has columns $r\!\YT{0.17in}{}{
			{{3}},
			{{\overline{3}}},
			{{\overline{1}}}}$, $r\!\YT{0.17in}{}{
			{{3}},
			{{\overline{1}}}}$ and $r\!\YT{0.17in}{}{
			{{\overline{1}}}}$. 
		Hence $K_+(T)=\YT{0.17in}{}{
			{{3},{3},{\overline{1}}},
			{{\overline{2}},{\overline{1}}},
			{{\overline{1}}}}$.
	Lemma \ref{sahpesjdt} shows that the column commutation action defined by the SJDT on two
	consecutive columns of a straight KN tableau $T$ of shape $\lambda$ gives rise to a permutohedron
	where the vertices are all the KN skew tableaux in the Knuth class of $T$ whose column
	length sequence is a permutation of the column length sequence of $T$ \cite{lascoux1990keys}. For instance, \eqref{jeudetaquin3} is a
	permutohedron (hexagon) for $\mathfrak{S}_3$. 

	Let $T=C_1C_2\cdots C_k$ be a straight KN tableau with columns $C_1, C_2, \dots,  C_k$. Note that, to compute which entries appear in the $i$-th column of $K_+(T)$ we do not need to look to the first $i-1$ columns of $T$. We only need the last column of a skew tableau obtained by applying the SJDT to the columns $C_i\dots C_k$ of $T$, so that the last column has the length of $C_i$, because by Lemma \ref{sahpesjdt} all last columns of skew tableaux associated to $T$ with the same length  
	are equal. Let $K_+^1(T)$ be the map that given a tableau returns the first column of $K_+(T)$.
	This is noticeable in \ref{jeudetaquin3} where $K_+(T)=K^1_+(C1C2C3)K^1_+(C2C3)K^1_+(C3)$. In general, $K_+(T)=K_+^1(C_1\cdots C_k)K_+^1(C_2\cdots C_k)\cdots K_+^1(C_k)$.
	Based on this observation and Lemma \ref{sahpesjdt}, next algorithm refines Definition \ref{rightkeymap} to compute $K_1^+(T)$ using SJDT:
	\begin{alg}\label{rightkjdtq}
		Let $T$ be a straight KN tableau:
		
		\textbf{1.} Let  $i=2$.
		
		\textbf{2.} If $T$ has exactly one column, return the right column of $T$. Otherwise, let $T_i:=T_2$ be the tableau formed by the first two columns of $T$. 
		
		\textbf{3.} If the length of the two columns of $T_i$ is the same, put $T_i':=T_i$. Else, play the SJDT on $T_i$ until both column lengths are swapped, obtaining $T_i'$.
		
		\textbf{4.} If $T$ has more than $i$ columns, redefine $i:=i+1$, and define $T_i$ to be the two-columned tableau formed with the rightmost column of $T'_{i-1}$ and the $i$-th column of $T$, and go back to $2.$. Else, return the right column of the rightmost column of $T_i'$.
	\end{alg}
	
	This algorithm is exemplified on the bottom path of \eqref{jeudetaquin3}.

	\begin{cor}
		If $T$ is a rectangular tableau, $K_+(T)=rC_k  rC_k\dots rC_k$ ($k$ times).
	\end{cor}
Next, we present a way of computing $K_+^1(T)$ that does not require the SJDT. Willis has done this when $T$ is a SSYT \cite{Willis2011ADW}. It is a simplified version of the algorithm presented here.
%
	
	\section{Right key - a direct way}\label{rightkeydw}
		Let $T = C_1C_2$ be a straight KN two column tableau and $spl(T)=\ell C_1rC_1\ell C_2rC_2$ a straight semistandard tableau. In particular, $rC_1\ell C_2$ is a semistandard tableau. The \emph{matching between $rC_1$ and $\ell C_2$} is defined as follows:
	
	$\bullet$  Let $\beta_1<\dots<\beta_{m'}$ be the elements of $\ell C_2$. Let $i$ go from $m'$ to $1$, match $\beta _i$ with the biggest, not yet matched, element of $rC_1$ smaller or equal than $\beta_i$.

	\begin{teo}[The direct way algorithm for the right key] \label{mainthm}
	Let $T$ be a straight KN tableau with columns $C_1, C_2, \dots, C_k$, and consider its split form $spl(T)$. For every right column $rC_2, \dots, rC_k$, add empty cells to the bottom in order to have all columns with the same length as $rC_1$. We will fill all of these empty cells recursively, proceeding from left to right. The extra numbers that are written in the column $rC_2$ are found in the following way:
	
	$\bullet$ match $rC_1$ and $\ell C_2$.
	
	$\bullet$ Let $\alpha_1<\dots<\alpha_m$ be the elements of $r C_1$. Let $i$ go from $1$ to $m$. If $\alpha_i$ is not matched with any entry of $\ell C_2$, write in the new empty cells of $rC_2$ the smallest element bigger or equal than $\alpha_i$ such that neither it or its symmetric exist in $rC_2$ or in its new cells. Let $C_2'$ be the column defined by $rC_2$ together with the filled extra cells, after ordering.

	To compute the filling of the extra cells of $rC_3$, we do the same thing, with $C_2'$ and $C_3$. If we do this for all pairs of consecutive columns, we eventually obtain a column $C_k'$, consisting of $rC_k$ together with extra cells, with the same length as $rC_1$. We claim that $C_k'=K_+^1(T)$.
	\end{teo}	
	\begin{ex}
		Let $T=C_1C_2C_3=\YT{0.17in}{}{
			{{1},{3},{\overline{1}}},
			{{3},{\overline{3}}},
			{{\overline{3}}}}$, with split form 	$spl(T)=\YT{0.17in}{}{
			{{1},{1},{2},{3},{\overline{1}},{\overline{1}}},
			{{2},{3},{\overline{3}},{\overline{2}}},
			{{\overline{3}},{\overline{2}}}}$.
		 We match $rC_1$ and $\ell C_2$, as indicated by the letters $a$ and $b$:
		$\YT{0.17in}{}{
			{{1},{1^a},{2^a},{3},{\overline{1}},{\overline{1}}},
			{{2},{3^b},{\overline{3}^b},{\overline{2}}},
			{{\overline{3}},{\overline{2}}}}$. Hence $\overline{2}$ creates a $\overline{1}$ in $rC_2$, completing the right column  $rC_2$:
	 $\YT{0.17in}{}{
			{{1},{1},{2},{3},{\overline{1}^a},{\overline{1}}},
			{{2},{3},{\overline{3}},{\overline{2}}},
			{{\overline{3}},{\overline{2}},,{{\color{blue}\overline{1}^a}}}}$.
		Now we match $C_2'$ and $\ell C_3$, which is already done, and see what new cells $3$ and $\overline{2}$ create in $rC_3$, obtaining
		$\YT{0.17in}{}{
			{{1},{1},{2},{3},{\overline{1}},{\overline{1}}},
			{{2},{3},{\overline{3}},{\overline{2}},,{{\color{blue}3}}},
			{{\overline{3}},{\overline{2}},,{{\color{blue}\overline{1}}},,{{\color{blue}\overline{2}}}}}$. Hence $K_+^1(T)=\YT{0.17in}{}{
			{{3}},
			{{\overline{2}}},{{\overline{1}}}}$ is obtained from $C'_3$ after reordering its entries.
	\end{ex}
	
	\subsection{The proof of Theorem \ref{mainthm}}
	It is enough to prove that by the end of this algorithm, the entries in $C_k'$ are the entries on the right column of the rightmost column of $T'_k$ from Algorithm \ref{rightkjdtq}.
	In fact, it is enough to do this for $k=2$. For bigger $k$ note that the entries that are "slid" into $C_k$ come from $rC_{k-1}$, so, to go to the next step on the SJDT algorithm we only need to know the previous right column, which is exactly what we claim to compute this way. The next lemma determines which number is added to $rC_2$ given that we know $\alpha$, the entry that is horizontally slid:

	\begin{lema}
		Suppose that $T=C_1C_2$ is a non-rectangular two-column tableau (if the tableau is rectangular then we have nothing to do). Play the SJDT on this tableau which ends up moving one cell from the first column to the second (some entries may change their values). Then,

$\bullet$ Immediately before the horizontal slide of the SJDT, the entry $\alpha$, on the left of the puncture,  is an unmatched cell of $rC_1$.

$\bullet$ Call $C_1'$ and $C_2'$ to both columns after the horizontal slide on $T$. The new entry in $rC_2'$, compared to $rC_2$, is the smallest element bigger or equal than $\alpha$ such that neither it or its symmetric exist in $rC_2$.
	\end{lema}
\begin{ex}
Let $T=\YT{0.17in}{}{
		{{2},{3}},
		{{3},{4}},
		{{5},{\overline{5}}},
		{{\overline{5}}},
		{{\overline{2}}}}$. After splitting, and just before the first horizontal slide, we have  $T=\YT{0.17in}{}{
		{{1},{2},{3},{3}},
		{{3},{3},{4},{4}},
		{{4},{5},{\overline{5}},{\overline{5}}},
		{{\overline{5}},{\overline{4}},{\ast}, {\ast}},
		{{\overline{2}},{\overline{1}}}}$. 
		The new entry in $rC_2$ is $\overline{2}$, as predicted by the lemma:
		 $\YT{0.17in}{}{
		{{1},{2},{2},{3}},
		{{3},{3},{3},{4}},
		{{4},{4},{\overline{5}},{\overline{5}}},
		{{\ast},{\ast},{\overline{4}},{\overline{2}}},
		{{\overline{2}},{\overline{1}}}}$. 
\end{ex}
	\begin{proof}
		\textbf{Case 1:} $\alpha$ is barred. Then $C_2'=C_2\cup\{\alpha\}$. If $\overline{\alpha}$  does  not exist neither in $C_2$ nor in $\Phi(C_2)$, then $\alpha$ will exist in both $C_2'$ and $\Phi(C_2')$. 
		If $\overline{\alpha}$ does exist in $C_2$, and consequently in $\Phi(C_2)$ (but $\alpha \notin \Phi(C_2)$), then $\alpha$ and $\overline{\alpha}$ will both exist in $C_2'$. Hence, in the construction of the barred part of $\Phi(C_2')$, compared to $\Phi(C_2)$, there will be a new barred number which is the smallest number bigger (or equal, but the equality can not happen) than $\alpha$ such that neither it nor its symmetric exist in the barred part of $\Phi(C_2)$ or the unbarred part of $C_2$ (i.e., $rC_2$). If $\alpha$ existed in $\Phi(C_2)$, then $\overline{\alpha}$ existed in $\Phi(C_2)$. That means that whatever number got sent to $\alpha$ in the construction of $\Phi(C_2)$ will be sent to the next available number, meaning that in $rC_2$ will appear a new number, the smallest number bigger (or equal, but the equality can not happen because $\alpha$ is already there) than $\alpha$ such that neither it nor its symmetric exist in $rC_2$.

		\textbf{Case 2:} $\alpha$ is unbarred. Then $C_2'=\Phi^{-1}(\Phi(C_2)\cup \{\alpha\})$. If $\overline{\alpha}$ does  not exist in $C_2$ nor in
		 $\Phi(C_2)$, then $\alpha$ will exist in both $C_2'$ and $\Phi(C_2')$. 
		If $\overline{\alpha}$ existed in $\Phi(C_2)$, and consequently in $C_2$, then both $\alpha$ and $\overline{\alpha}$ will exist in $\Phi(C_2')$, hence, if we start in the coadmissible column, in the construction of the unbarred part of $C_2'$, compared to $C_2$, there will be a new unbarred number which is the smallest number bigger than $\alpha$ such that neither it nor its symmetric exist in $rC_2$. Finally, if $\alpha$ existed in $C_2$, then $\overline{\alpha}$ also existed in $C_2$.
		 That means that whatever number got sent to $\alpha$ in the construction of $C_2$, from $\Phi(C_2)$, will be sent to the next available number, meaning that in $rC_2$ will appear a new number, the smallest number bigger than $\alpha$ such that neither it nor its symmetric exist in $rC_2$.
	\end{proof}
	
\begin{proof}[Proof of Theorem \ref{mainthm}]
Each SJDT in $T$, a two-column skew tableau, moves a cell from the first to the second column. We will prove that if we apply the direct way algorithm after each SJDT, the output $C_2'$ does not change. The cells on $\ell C_2$ without cells to its left do not get to be matched. When we slide horizontally, the columns $rC_1$ and $\ell C_2$ may change more than the adding/removal of $\alpha$, the horizontally slid entry.	
	Since the horizontal slides happen from top to bottom, we only need to see  what changes happen to bigger entries than the one slid. 
All entries above $\alpha$ are matched to the entry in the same row in $\ell C_2$.
	
	If $\alpha$ is barred then, the remaining barred entries of $rC_1$ and $\ell C_2$ remain unchanged, and since all entries above $\alpha$, including the unbarred ones, are matched to the entry directly on their right, there is no noteworthy change and everything runs as expected.
	
	If $\alpha$ is unbarred then, the remaining unbarred entries of $rC_1$ and $\ell C_2$ remain unchanged.
	In the barred part of $rC_1$ either nothing happens, or there is an entry bigger than $\overline{\alpha}$, $\overline{x}$, that gets replaced by $\overline{\alpha}$. Note that $\overline{x}$ must be such that for every number between $\overline{x}$ and $\overline{\alpha}$, either it or its symmetric existed in $rC_1$.
	In the barred part of $\ell C_2$, if $\overline{\alpha}\in \ell C_2$, then $\overline{\alpha}$ gets replaced by $\overline{y}$, smaller than $\overline{\alpha}$, such that for every number between $\overline{y}$ and $\overline{\alpha}$, either it or its symmetric existed in $\ell C_2$, and both $y$ and $\overline{y}$ do not exist in $\ell C_2$. 
	
	Let's look to $\ell C_2$. Let $\alpha<p_1<p_2<\dots< p_m=y$ be the numbers between $\alpha$ and $y$ that does not  exist in $\ell C_2$, right before the horizontal slide. Then, their symmetric exist in $\ell C_2$. For all numbers in $rC_2$ between $\alpha$ and $y$, exist, in the same row in $r C_1$, a number between $\alpha$ and $y$. Let $\alpha<p_1'<p_2'<\dots< p_m'= y$ be the missing numbers between $\alpha$ and $y$ in $rC_1$, then $p_i\leq p_i'$.
	Note that $\overline{p_1}>\overline{p_2}>\dots >\overline{p_m}=\overline{y}$ exist in $\ell C_2$ after the horizontal slide and that the biggest numbers between $\overline{\alpha}$ and $\overline{y}$ (not including $\overline{\alpha}$) that can exist in $rC_1$ are $\overline{p_1'}>\overline{p_2'}>\dots> \overline{p_m'}$, and since $\overline{p_i}\geq \overline{p_i'}$, the matching holds for this interval after swapping $\overline{\alpha}$ by $\overline{y}$ in $\ell C_2$.

	Now let's look to $rC_1$. Before the slide, call $\overline{x'}$ to the biggest unmatched number of $rC_1$ smaller or equal then $\overline{x}$. If no such $\overline{x'}$ exists, then everything in $rC_1$ between $\overline{\alpha}$ and $\overline{x}$ is matched, hence swapping $\overline{x}$ by $\overline{\alpha}$ will keep all of them matched, meaning that the algorithm works in this scenario.
	Let $x'<q_1<q_2<\dots< q_m<\alpha$ be the numbers between $x'$ and $\alpha$ that does not  exist in $r C_1$, right before the horizontal slide. Then, their symmetric exist in $r C_1$. For all numbers in $r C_1$ between $x'$ and $\alpha$, exist, in the same row in $\ell C_2$, a number between $x'$ and $\alpha$, because $\alpha$ is unmatched. Let $x'<q_1'<q_2'<\dots< q_m'<\alpha$ be the missing numbers between $x'$ and $\alpha$ in $\ell C_2$, then $q_i\geq q_i'$.
	Note that $\overline{q_1}>\overline{q_2}>\dots >\overline{q_m}>\overline{\alpha}$ exist in $r C_1$ after the horizontal slide and the numbers between $\overline{x'}$ and $\overline{\alpha}$ that can exist in $\ell C_2$ are $\overline{q_1'}>\overline{q_2'}>\dots> \overline{q_m'}$, and since $\overline{q_i}\leq \overline{q_i'}$, these numbers are matching a number bigger or equal then $q_i$  in $r C_1$, meaning that $\alpha$ is unmatched in $rC_1$.
	Ignoring signs, the numbers that appear in either $rC_1$ or $\ell C_2$ are the same. So before playing the SJDT, applying the direct way algorithm we have that the unmatched numbers in $rC_1$ are sent to the not used numbers of $\overline{q_1'}>\overline{q_2'}>\dots> \overline{q_m'}$ in $\ell C_2$ (this is a bijection), and $\overline{x'}$ is sent to the smallest available number, bigger or equal than $\overline{x'}$. Now consider $rC_1$ and $\ell C_2$ after the slide. In $rC_1$ we replace $x'$ by $\overline{\alpha}$ and remove $\alpha$ and in $\ell C_2$ there is $\alpha$ or $\overline{\alpha}$. In the direct algorithm, all unmatched numbers of $\overline{q_1}>\overline{q_2}>\dots >\overline{q_m}>\overline{\alpha}$ are sent to the not used numbers of $\overline{q_1'}>\overline{q_2'}>\dots> \overline{q_m'}$ in $\ell C_2$, but now we have more numbers in the first set than in the second, meaning that $\overline{\alpha}$ will bump the image of the least unmatched number, which will bump the image of the second least unmatched number, and so on, meaning that the image of biggest unmatched will be out of this set. This image will be the smallest number available, which was the image of $x'$ before the horizontal slide.
	
	Hence, the outcome of the direct way does not change due to the changes to the columns when we play the SJDT, meaning that the outcome is what we intend.
\end{proof}
\begin{obs}
	It is possible to modify this algorithm in order to compute left keys.
\end{obs}
\emph{I am grateful to O. Azenhas, my Ph.D. advisor, for her help on the preparation of this paper.}

\printbibliography

\end{document}